\documentclass{amsart}
\usepackage{color,graphicx,enumerate,wrapfig,amssymb}
\usepackage{amssymb} 
\usepackage{amsmath}
\usepackage{esint}
 \usepackage{cite}
\usepackage{comment}
\usepackage{color,graphics}
\usepackage{fixme}
\usepackage{graphicx} 
\usepackage{soul} \usepackage[normalem]{ulem}
\newtheorem{thm}{Theorem}[section]

\newtheorem{lem}[thm]{Lemma}
\newtheorem{prop}[thm]{Proposition}
\theoremstyle{definition}

\theoremstyle{remark}
\newtheorem{rem}[thm]{Remark}
\numberwithin{equation}{section}

\newcommand{\R}{\mathbb R}

\newcommand{\eps}{\epsilon}

\newcommand{\p}{\partial}

\theoremstyle{definition}

\theoremstyle{remark}

\numberwithin{equation}{section}
\theoremstyle{claim}

\thanks{DD is supported by a NSF grant (RTG 1937254). DJ is supported by a Simons Foundation Grant (601948 DJ). HS is supported by Swedish Research Council}

\begin{document}
\title[Inhomogeneous global minimizers]{Inhomogeneous global minimizers to the one-phase free boundary problem}
\author{Daniela De Silva}
\address{Department of Mathematics, Barnard College, Columbia University, New York, NY, 10027}
\email{\tt  desilva@math.columbia.edu}
\author{David Jerison}
\address{Department of Mathematics, Massachusetts Institute of Technology, Cambridge MA,  }
\email{jerison@math.mit.edu}

\author{Henrik Shahgholian}
\address{Department of Mathematics, KTH  Royal Institute of Technology, Stockholm, Sweden}
\email{henriksh@math.kth.se}
\begin{abstract} Given a global 1-homogeneous minimizer $U_0$ to the Alt-Caffarelli energy functional, with $sing(F(U_0)) = \{0\}$, we provide a foliation of the half-space $\R^{n} \times [0,+\infty)$ with dilations of graphs of global minimizers $\underline U \leq U_0 \leq \bar U$ with analytic free boundaries at distance 1 from the origin.
\end{abstract}

\maketitle


\section{Introduction}
\subsection{Background}
In this paper we are interested in minimizers to the energy functional,
\begin{equation}\label{J}
J(u) = J_\Omega(u)= \int_\Omega  ( |\nabla u|^2 + \chi_{\{u > 0\}})\;dx ,
\end{equation}
where  $\Omega$ is a domain in $\R^n$  ($n\geq 2$) and $u\geq 0.$ Minimizers of $J$ were first investigated systematically by Alt and Caffarelli. 
Two fundamental properties are proved in the pioneering article \cite{AC}, namely, the Lipschitz regularity of minimizers and the regularity of ``flat'' free boundaries.  These in turn, give the almost-everywhere regularity of minimizing free boundaries. 
The viscosity approach to the regularity of the associated free boundary problem 
\begin{equation}\label{visc}\begin{cases}
\Delta u=0 \quad \text{in $\Omega^+(u):= \Omega \cap \{u>0\}$}\\
|\nabla u|=1 \quad \text{on $F(u):= \p \{u>0\} \cap \Omega,$}
\end{cases}\end{equation}
was later developed by Caffarelli in \cite{C1,C2}. There is a wide literature on this problem and the corresponding two-phase problem, and for a comprehensive treatment we refer the reader to \cite{CS} and the references therein.

Here we are interested in one-phase global minimizers, that is, minimizers $u \geq 0$ of $J$ over all balls $B_R \subset \R^n$ among functions that agree with $u$ on $\partial B_R$.
 In dimension $n=2$, Alt and Caffarelli \cite{AC} showed that (up to rotation) the only  global minimizer of $J$ is $x_n^+$. The same result was obtained by Caffarelli, Jerison and Kenig in dimension $n=3$ \cite{CJK}, and by Jerison and Savin in dimension $n=4$ \cite{JS}. The classification of global minimizers implies the smoothness of minimizing free
boundaries  in both the one-phase and two-phase problem in dimension $n\leq 4$. These results rely on the flatness theorem and on the Weiss Monotonicity formula (see Section 2 for the precise statement)  which allows one to consider only the case of 1-homogeneous minimizers. It is  conjectured that the results above remain true up to dimension $n\leq 6$. On
the other hand De Silva and Jerison provided in \cite{DJ} an example of a singular 1-homogeneous minimal solution in dimension $n = 7$. In \cite{H},  Hong studied the stability of Lawson-type cones for \eqref{visc} in low dimensions and showed that in dimension $n = 7$ there is another stable cone besides the minimizer in \cite{DJ}. 
%
%
%

\subsection{Main result}

Let $U_0 \geq 0$ be a global energy minimizer to $J$, and assume that $U_0$ is homogeneous of degree 1, and $F(U_0) \setminus \{0\}$ is an analytic hypersurface, while $0$ is a singular point (hence, by the discussion above, $n\geq 5$). In this paper we construct smooth inhomogeneous global minimizers asymptotic to $U_0$ at infinity. 
In fact, $U_0$ is trapped between two global smooth inhomogeneous minimizers obtained by perturbing away the singularity, and whose dilations foliate the whole space. 
The existing theory of  one-phase minimizers establishes a strong resemblance between the theory of free boundaries and the theory of minimal  surfaces. Our result reaffirms this similarity, providing an analogue of Hardt and Simon's result in the context of area minimizing cones \cite{HS}. 

In order to state our main theorem precisely, we recall some known facts and refer to \cite{JS} for further details on the linearized problem (see also Section 3). Let 
$H>0$ denote the mean curvature of $\p \{U_0>0\}$ oriented towards the complement of the connected set $\{U_0>0\}$,  and 
 consider the linearized problem associated to $U_0$: 
\begin{equation}\label{wintro}\begin{cases}
\Delta w=0 \quad \text{in $\{U_0>0\}$}\\
\p_\nu w + Hw=0 \quad \text{on $F(U_0) \setminus \{0\}$},\\
\end{cases}\end{equation} with $\nu$ the interior unit normal to $F(U_0).$
Let  $w(x):= |x|^{-\gamma} \bar v(\theta)$ with $\bar v$ the first eigenfunction of the Laplacian on $\mathbb S^{n-1} \cap \{U_0>0\}$, 
$$\Delta_{\mathbb S^{n-1}} \bar v= \lambda \bar v, \quad \text{on $\mathbb S^{n-1} \cap \{U_0>0\}$,}$$
with Neumann boundary condition $$\p_\nu\bar v + H \bar v=0, \quad \text{on $F(U_0) \cap \mathbb S^{n-1}.$}$$ Then $\bar v>0$ on  $\mathbb S^{n-1} \cap \overline{\{U_0>0\}}$ and $\lambda>0.$
It is easily computed (see Section 4) that  $w$ solves \eqref{wintro}
as long as $\gamma=\gamma_\pm \in \R, \gamma_+\geq \gamma_->0,$ satisfy $\gamma^2 - (n-2)\gamma + \lambda=0$. If $\gamma_- \neq \gamma_+$ we call
\begin{equation}\label{defv}V_{\gamma_\pm}(x) := |x|^{-\gamma_{\pm}} \bar v.\end{equation}
By abuse of notation, if $\gamma_-=\gamma_+$ and we call $\gamma_0$ this common value, we set 
\begin{equation}\label{defv2}V_{\gamma_-}(x):= |x|^{-\gamma_0} (\ln |x|+1) \; \bar v, \quad V_{\gamma_+}(x)= |x|^{-\gamma_0} \; \bar v.\end{equation}

Our main result reads as follow. As usual, $B_R(x_0) \subset \R^n$ denotes the ball of radius $R>0$ and center $x_0$, and when $x_0=0$ we drop the dependence on it. Also, given a function $u \geq 0$ defined on $\R^n$, we define for $t>0$
$$u_t(x)= \frac{1}{t}u(tx), \quad \Gamma(u_t):= \{(x, u_t(x)): x \in \overline{\{u_t>0\}}\}.$$
 Finally, constants that depend only on the ingredients i.e. $n, U_0,$ are called universal.

\begin{thm}\label{main} Let $U_0$ be as above. There exist $\bar U, \underbar U$ global minimizers to \eqref{J}, such that 
\begin{enumerate}
\item $\underbar U \leq U_0 \leq \bar U,$ and $dist(F(\bar U), \{0\}) = dist(F(\underbar U), \{0\}) =1;$\\ \item $F(\bar U), F(\underbar U)$ are analytic hypersurfaces;\\ 
\item there exist universal constants $R>0$ large, $\alpha'>0$ small, such that if $\gamma_- \neq \gamma_+$,
 $$\bar U(x) =U_0(x)+ \bar a V_{\gamma_-}(x)+ O(|x|^{-\gamma_- - \alpha'}), \text{in $\{U_0>0\} \setminus B_R$},$$  or $$ \bar U(x)=U_0(x) + \bar a V_{\gamma_+}(x)+ O(|x|^{-\gamma_+ - \alpha'}), \text{in $\{U_0>0\} \setminus B_R$,}$$ for some $\bar a>0$ universal.
 The analogous statement for $\underline U$ holds in $\{\underline U >0\} \setminus B_R$ with $\underline a <0$.
 If $\gamma_-=\gamma_+$, then
  $$\bar U(x) =U_0(x)+ \bar a V_{\gamma_-}(x)+ \bar b V_{\gamma_+}+ O(|x|^{-\gamma_0 - \alpha'}), \text{in $\{U_0>0\} \setminus B_R$},$$ with $\bar a \geq 0$ and if $\bar a=0$ then $\bar b>0.$  The analogous statement for $\underline U$ holds in $\{\underline U >0\} \setminus B_R$ with $\underline a \leq 0$ and $\underline b<0$ whenever $\underline a=0;$
 \\ 
\item for any $x \in \{U_0=0\}^o$, the ray $\{t x$, $t>0\}$, intersects $F(\bar U)$ in a single point and the intersection is transverse; similarly, for any $x \in  {\{U_0>0\}}$, the ray $\{t x$, $t>0\}$, intersects $F(\underline U)$ in a single point and the intersection is transverse; \\
\item the graphs  $\bar \Gamma_t:= \Gamma(\bar U_t), \underline \Gamma_t:=\Gamma(\underline{U_t})$ foliate the half-space $\R^{n} \times [0,+\infty)$, i.e., 
$$\R^{n} \times [0,+\infty) = \bigcup_{t \in (0,\infty)} (\bar \Gamma_t \cup \underline \Gamma_t);$$\\
\item if $V$ is a global minimizer to $J$ and  $V \geq U_0$ (resp. $V \leq U_0$), then $$V \equiv \bar U_t, \quad (\text{resp. $V \equiv \underline U_t$}),$$ for some $t>0,$ unless $V \equiv U_0.$ 
\end{enumerate}
\end{thm}

In fact, the uniqueness property  $(vi)$ holds more generally for global critical points that satisfy a uniform non-degeneracy condition (see Section 2 for this notion). Furthermore the expansion in $(iii)$ implies the same expansion for the free boundaries $F(\bar U), F(\underline U)$, as graphs over the cone $F(U_0)$ in the outer normal direction $-\nu$ (see Remark \ref{expfb}). 

\subsection{Further background}

In the context of critical point,  Hauswirth, Helein, and Pacard  \cite{HHP} discovered an explicit  family of simply-connected planar regions (so-called exceptional domains)
$$\Omega_a := \{(x_1, x_2) \in \R^2 : |x_1/a| < \pi/2 + \cosh(x_2/a)\}, a > 0$$
whose boundary consists of two curves (hairpins), and a positive harmonic function $H_a(x) = aH_1(x/a)$ on $\Omega_a$ that satisfies the free boundary conditions $H_a = 0$ and $|\nabla H_a |=1$ on $\p \Omega_a$. Extending $H_a$ to be zero in the complement of $\Omega_a$, we have a non-trivial entire solution to \eqref{visc} (an explicit formula for $H_a$ is given using conformal mappings). The graphs of the functions $H_a$ give a foliation of space using dilates of an unstable critical point of the functional.  In \cite{JK}, Jerison and Kamburov characterized blow-up limits of classical solutions to \eqref{visc} in the disk, with simply-connected positive phase, as either (a) half-plane, (b) two-plane or (c) hairpin solutions. This relates to previous results of classifications of entire solutions with simply-connected positive phase due to Khavinson, Lundberg and Teodorescu \cite{KLT} and Traizet  \cite{T}. Traizet showed that classical entire solutions satisfying that no connected component of $F(u)$ is compact in the open disk must be one of the forms (a), (b), or (c). Khavinson et al. (2013) showed that the same conclusion is true under a natural, weak regularity assumption on the free boundary known as the Smirnov property.

\subsection{Organization of the paper}

The paper is organized as follows. In Section 2 we provide some preliminary properties of minimizers. The following section is devoted to the proof of Theorem \ref{main}. The core of the proof is the analysis of the asymptotic behaviour of a solution of a perturbation of the linearized problem \eqref{wintro}, using appropriate families of subsolutions and supersolutions. This analysis is carried on in Section 4. Finally, the Appendix contains a technical result (relying on the Hodograph transform) needed in the proof of the main theorem.

\smallskip

When we announced this result, M. Engelstein and coauthors informed us that they have 
obtained what appear to be essentially similar results to the ones in this article in a paper under preparation \cite{ESV}.  

\section{Preliminaries}

In this section we recall some basic properties of minimizers and we obtain a few preliminary results which we use in Section 3 towards the proof of Theorem \ref{main}. Throughout the paper, for a non-negative function $u$ on a domain $\Omega$, we use the notation 
$$ F(u):= \Omega \cap \p \{u>0\}.$$

We start by a quick recap of the classical theory for minimizers of \eqref{J}, which we will be using throughout the paper (see \cite{AC}). First of all, we say that $u \in H^1_{loc}(\Omega)$ minimizes $J$ in $\Omega$ if on any smooth compact set $\Omega' \subset \Omega,$
$$J(u) \leq J(u + v), \quad \forall v \in H_0^1(\Omega').$$

For a given boundary data $\varphi \geq 0$ with finite energy, there exists a non-negative minimizer $u \in H^1(\Omega)$ of $J$ such that $u-\varphi \in H_0^1(\Omega).$ Moreover, $u \in C^{0,1}(\Omega)$ and satisfies \eqref{visc} in the viscosity sense (see \cite{C2} for the definition of viscosity solution and a proof of this claim). Finally, $u$ satisfies a strong non-degeneracy property, that is for any $x_0 \in F(u)$, $\sup_{B_r(x_0)} u \geq c r$ for all balls $B_r(x_0) \subset \Omega$. Lipschitz continuity and non-degeneracy are the key to the following compactness property (see Lemma 4.7 and Lemma 5.4 in \cite{AC}).

\begin{lem}\label{compact} Let $\{u_k\} \in H^1_{loc}(\Omega)$  be a sequence of minimizers to $J$ in $\Omega$ with 
$u_k \to u$  uniformly on compact subsets, then \begin{itemize}
\item $\{u_k > 0\} \to \{u > 0\}$ 
and $F(u_k) \to F(u)$ locally in the Hausdorff distance;
\item $\chi_{\{u_k>0\}} \to \chi_{\{u>0\}}$ in $L^1_{loc}(\Omega)$; 
\item $\nabla u_k \to \nabla u$ a.e. in $\Omega.$
\end{itemize} Moreover, $u$ minimizes $J$ in $\Omega$.
\end{lem}

We also recall Weiss Monotonicity Formula \cite{W}.

\begin{thm} \label{MF} If $u$ is a minimizer to $J$ in $B_R,$ then $$\Phi_u(r) := r^{-n}J_{B_r}(u)- r^{-n-1}\int_{\p B_r} u^2, \quad 0<r \leq R,$$ is increasing in $r.$ Moreover $\Phi_u$ is constant if and only if $u$ is homogeneous of degree $1$.\end{thm}

We now notice that,
$$J(\min\{u,w\}) + J(\max\{u,w\}) = J(u) + J(w),$$
from which we deduce that if $u,w$ minimize $J$ in $\Omega$ and $u\geq w$ on $\p B_1$, then $\min\{u,w\},\max\{u,w\}$ minimize $J$ in $\Omega$ as well (with boundary data $w$ and $u$ respectively.).
We can then provide a comparison principle.

\begin{prop}\label{comp1} 
Let $u,w$ minimize $J$ in $B_1$ and assume that for $0<\eps<1$, $$u \geq w \quad \text{on $B_1 \setminus B_{1-\eps}$}.$$ Then $u \geq w $ in $B_1$.
\end{prop}

\begin{proof} Since $u \geq w$ on $\p B_1$, we have that $v:=\min\{u,w\}$ minimizes $J$ among competitors with boundary value $w$. However, by our assumptions, $v \equiv w$ in $B_1 \setminus B_{1-\eps}$, hence by the harmonicity of minimizers in their positivity set, we conclude that $v \equiv w$ in $D \cap B_1$ for any connected component $D$ of $\{w>0\}$ which intersect the annulus. By the maximum principle, any connected component of $\{w>0\}$ will intersect the annulus,  hence $v\equiv w$ in $B_1.$ 
\end{proof}

\begin{prop}\label{comp2}
Let  $u,w$ minimize $J$ in $B_1$ and assume that $$u \geq w \quad \text{on $\p B_1$}, \quad u(x_0)> w(x_0) \quad \text{at $x_0 \in \p B_1$}.$$ If $\{w>0\} \cap \p B_1$ is connected, then $u \geq w $ in $B_1$.
\end{prop}
\begin{proof} We argue as in the proof of the lemma above, however in this case we only have that $\min\{u,w\} \equiv w$ in a neighborhood of $x_0.$ Thus they coincide in the connected component of $\{w>0\}$ which includes this neighborhood, and the conclusion follows from the connectivity assumption.
\end{proof}

We now obtain a uniqueness result. Recall that $\{U_0>0\}$ is connected \cite{CJK}.

\begin{lem}[Uniqueness]\label{uni} Let $w$ be a minimizer of \eqref{J} in $B_1$ such that $w=U_0$ on $\p B_1$. Then $w \equiv U_0$ in $B_1.$
\end{lem}
\begin{proof} Let $$W= w \quad \text{in $B_1$}, \quad W=U_0 \quad \text{in $B_2 \setminus B_1$}.$$ 
Since $w- U_0 \in H^1_0(B_1),$ we have $W - U_0 \in H^1_0(B_2)$ and $W$ and $U_0$ both minimize $J$ in $B_2$ 
among competitors with boundary data $U_0$. Hence,
$$\Delta W = 0 \quad  \hbox{in } B_2 \cap \{ W > 0\},  
\qquad \Delta U_0 = 0 \quad \hbox{in } B_2 \cap \{ U_0 > 0\}. 
$$
In particular, since $W=U_0$ in $B_2\setminus B_1$, by unique continuation $w$ and $U_0$ must agree on $B_1 \cap \overline{ \{U_0>0\}}.$ If $w(\bar x)>0$ at a point $\bar x \in B_1 \cap \{U_0 =0\}^o$, then $J_{B_1}(w) > J_{B_1}(U_0)$ in $B_1$, contradicting minimality. 
\end{proof}

Finally, we prove a  strict separation lemma.

\begin{lem}\label{separation} Let $w$ be a minimizer to \eqref{J} in $B_1$ such that $w \geq U_0$ on $\p B_1$. Then $w \geq U_0$ in $B_1$. Moreover, $w > U_0$  in $B_1 \cap \overline{\{U_0>0\}}$, 
 unless $w \equiv U_0$ in $B_1 \cap \overline{\{U_0>0\}}$.
\end{lem}
\begin{proof} Since $w \geq U_0$ on $\p B_1$, $\min\{w, U_0\}$ is also a minimizer to $J$ in $B_1$ with boundary data $U_0.$ By the previous uniqueness result, we deduce that $\min\{w,U_0\} \equiv U_0$ in $B_1$ and the first part of our lemma is proved. Assume $w \not \equiv U_0$ in $B_1 \cap \overline{\{U_0>0\}}.$ By the maximum principle, we immediately deduce that $w>U_0$ in $ B_2 \cap \{ U_0> 0\}$ 
Moreover, if $w$ touches $U_0$ by above at a free boundary point $x_0 \neq 0$, since $\p_\nu U_0 (x_0)= \p_\nu w(x_0)=1$, we contradict Hopf's boundary point lemma. In particular,
$$w > U_0 \quad \text{in $(\bar B_{3/4} \setminus B_{1/2})\cap \overline{\{U_0>0\}},$}$$ implying that for $\delta$ small,
  $$w(x) \geq U_0(x+y), \quad |y|\leq \delta, \quad \text{in $\bar B_{3/4} \setminus B_{1/2}$},$$ and 
  by Proposition \ref{comp1} this is true in $B_{3/4}$. 
  Choosing $y$ so that $U_0(y)>0$, we get that $w(0)>0$, concluding the proof of the desired claim.
 \end{proof}

\section{The proof of Theorem \ref{main}}

This section is devoted to the proof of our main Theorem \ref{main}. The first subsection deals with the existence, while the second one handles the regularity.

\subsection{Existence and basic properties.} In this subsection we prove part $(i)$ of Theorem \ref{main}, that is the existence of $\bar U, \underline U$, and provide some key technical lemmas about their vertical distance from $U_0$. We start with the existence.

\medskip

\textit{Proof of Theorem $\ref{main}$-$(i)$.} First, we construct the global minimizer $\bar U$.
Let $x_0 \in \p B_1 \cap \{U_0>0\}$ and let $g \geq 0$ be a smooth function compactly supported in a neighborhood of $x_0$ on $\p B_1.$
 For $\epsilon >0,$ we define  the boundary value 
 $g_\epsilon := U_0 + \epsilon g \geq U_0$ on $\partial B_1$, $g_\eps(x_0)> U_0(x_0),$ and call $u_\eps$ a minimizer to $J_{B_1}$ with this boundary value. Notice that since $g_\eps$ is H\"older continuous, $u_\eps$ is uniformly H\"older continuous up to the boundary and it achieves the boundary data continuously. By Lemma \ref{separation}, $0 \not \in F(u_\eps)$ and $u_\eps \geq U_0$ for all $\eps>0.$
Let $\epsilon_j \searrow 0$ (as $j  \to\infty$), and for each $j$, let $B_{r_j}  \subset \{u_{\epsilon_j} > 0 \}$ be the largest ball included in the positive phase of $u_{\eps_j}$. Notice that by Proposition \ref{comp2}, $u_{\eps_j}$ is a decreasing sequence. By the Lipschitz continuity and the compactness Lemma \ref{compact}, we conclude that up to extracting a subsequence, $u_{\eps_j}$ converges to a minimizer $v$ of $J$ in $B_1$, and since $g_{\eps_j} \to U_0$, we conclude by the uniqueness Lemma \ref{uni} that $v=U_0$ and in particular $r_j \to 0$ and $j \to \infty$.  
 Consider now for each $j$ the rescaling $\bar U_j (x) = u_{\epsilon_j}(r_j x) / r_j$. Then $\bar U_j$ is a  minimizer  in $B_{1/r_j}$, $B_1 \subset \{ \bar U_j >0\}$, and $\partial B_1 \cap F(\bar U_j) \not = \emptyset$.  Moreover, $\bar U_j (x) \leq C(1 + |x|)$  in $B_{\frac{1}{2r_j}} $, where $C$ is independent of $j$, since each $u_{\eps_j}$ is uniformly Lipschitz in $B_{1/2}$. Again by compactness, up to extracting a subsequence, $\bar U_j \to \bar U$ and $\bar U$ is a global minimizer with $B_1 \subset \{\bar U > 0\},$ and universal Lipschitz bound.
 Finally by construction $\bar U \geq U_0$ as desired.

Similarly, $\underline  U$ can be build starting from a family of minimizers with boundary data,  $0\leq g_\epsilon := U_0 - \epsilon g \leq U_0$. \qed

 \medskip

 Next, we prove the following asymptotic behavior for $\bar U, \underline U.$
   
\begin{lem}\label{asy}$\bar U,\underline U$ are asymptotic to $U_0$ at infinity, i.e.,  $$\lim_{t \to \infty} \bar U_t(x) = \lim_{t \to \infty} \underline U_t(x) =U_0(x),$$ uniformly on compact sets in $\R^n$.\end{lem}
\begin{proof} We prove the statement for $\bar U$. The same argument also applies to $\underline U$. For $t_k \to \infty$ as $k \to \infty$, consider the sequence of rescalings, $\bar U_{t_k}(x): = \frac{\bar U(t_kx)}{t_k}$ which in view of the equi-Lipschitz continuity and the compactness Lemma \ref{compact}, will converge to a global minimizer $V \geq U_0$ (up to extracting a subsequence). On the other hand, again by the Lipschitz continuity,
 $$\Phi_V(r)=\lim_{k \to \infty} \Phi_{\bar U_{t_k}}(r)= \lim_{k \to \infty}\Phi_{\bar U}(t_kr)= a \in \R, \quad \forall r>0,$$
 and by Theorem \ref{MF}, $V$ is homogeneous of degree 1.
 
Thus, the restrictions of $V, U_0$ to the unit sphere solve the eigenvalue problem, 
$$\Delta_{\mathbb S^{n-1}} u=- \lambda_1 u \quad \text{on $\{u>0\} \cap \mathbb S^{n-1}$},$$ for $\lambda_1= n-1$. However, since $\{U_0>0\} \subset \{V>0\}$ and eigenvalues are monotone with respect to the inclusion of domains, we deduce from standard arguments that $U_0 \equiv V.$
\end{proof}

Finally, using Lemma \ref{asy} and Proposition \ref{nearby} in the Appendix, we can deduce the following result, which is key to our strategy. Here and henceforth, we denote by $E^c:= \R^n \setminus \bar E.$ Also, constants depending on $n, U_0$ will be called universal and may vary from one appearance to another in the body of the proofs. For notational simplicity, the positivity set of $U_0$ is called $\Omega_0:=\{U_0>0\}$, and as pointed out in the introduction,  $-(\p_{\nu\nu}U_0)=H>0$ is the mean curvature of $\p \Omega_0$ oriented towards the complement of the connected set $\Omega_0.$ 

To fix ideas, from now on, we prove the statements involving $\bar U$.

\begin{lem}\label{linearized1}There exists universally large $R>0$, such that $v:= \bar U -U_0 > 0$ satisfies
\begin{equation}\label{veq}\begin{cases}
\Delta v=0 \quad \text{in $\Omega_0 \cap B^c_{R}$},\\
\p_\nu v +H(1+ O(\frac{v}{|x|}))v=0 \quad \text{on $\p \Omega_0 \cap B^c_{R}$},\\
\end{cases}\end{equation} with $\nu$ the interior unit normal to $\p \Omega_0$.
Moreover, 
$$O\left(\frac{v}{|x|}\right)= o(1), \quad \text{as $|x| \to \infty$}.$$
\end{lem}
\begin{proof}
Let $\bar U_s(x):=s^{-1}\bar U(sx), s\geq 1$. We prove that  $v_s:= \bar U_s- U_0$
satisfies 
\begin{equation}\label{le}\begin{cases}
\Delta v_s=0 \quad \text{in } \Omega_0 \cap A,\\
\p_\nu v_s + H  v_s = O(|v_s|^2) \quad \text{on $\p \Omega_0 \cap A$},\\
\end{cases}\end{equation}
with $\nu$ the interior unit normal to $\p \Omega_0$, the constant in $O(|v_{s}|^2)$ universal, and $A$ an open annulus included in $B_2 \setminus \bar B_1$. Below $A_l:= B_{2-l} \setminus \bar B_{1+l}$, $l<1$.

Since $\bar U_s, U_0$ are harmonic in their positivity set, and $\bar U_s \geq U_0$, the harmonicity of $v_s$ in $\Omega_0$ is obvious. In order to prove the boundary condition we proceed as follows. First, in view of Lemma \ref{asy} 
\begin{equation*}\|\bar U_s - U_0\|_{L^\infty(B_{2} \setminus B_1)} \to 0, \quad \text{as $s\to \infty$},\end{equation*} thus, by the flatness implies regularity results \cite{C2}, there exists $\eps_0$ universal such that if $s=s(\eps_0)$ is large enough so that
\begin{equation}\label{under}\|\bar U_s - U_0\|_{L^\infty(B_{2} \setminus B_1)} \leq \epsilon_0\end{equation}
then in $A_{1/5} \cap \{\bar U_s> 0\}$,  
$\|\bar U_s\|_{C^3} \leq C$ with $C>0$ universal. Moreover, in a sufficiently small ball $B_\rho(x), x \in F(U_0)\cap A_{1/5}$, $\bar U_s$ and $U_0$ satisfy Proposition \ref{nearby} (in view of the convergence of $\bar U_s$ to $U_0$ and non-degeneracy), hence by a covering argument we obtain that in $A_{2/5} \cap \bar \Omega_0,$ 
\begin{equation}\label{hodo}
\|\bar U_s- U_0\|_{2,\alpha} \leq C \|\bar U_s - U_0\|_\infty, 
\quad \bar U_s - U_0 \sim \eps:=\|\bar U_s - U_0\|_\infty, 
\end{equation}
 $s$ large, and $\eps=\eps(s) \ll \eps_0$.

Let us define the function $\psi_s >  0$, such that 
\begin{equation}\label{psi}x \in F(U_0) \cap A_{2/5} \longrightarrow x-\psi_s(x) \nu_x \in F(\bar U_s).\end{equation} In view of \eqref{hodo}, since $\p_\nu U_0(x)= 1$ on $F(U_0)$, we have that $\p_\nu \bar U_s(x) =1 + O(\eps)$ on $F(U_0) \cap A_{2/5}$. Hence, $ F(\bar U_s)$ is a graph over $F( U_0)$ locally in $A_{2/5}$, $\psi_s$ is well defined, it is bounded by $C\eps$, 
and in fact \begin{equation}\label{expexp}\psi_s(x) = v_s(x) + O(\eps^2).\end{equation}
Then, 
\begin{align*} \nabla \bar U_s (x-\psi_s(x)\nu_x) &= \nabla \bar U_s(x)
 -v_s(x)D^2 \bar U_s(x) \nu_x + O(\eps^2)\\   
&= \nu_x -v_s(x)D^2 \bar U_s(x) \nu_x + \nabla v_s(x) +O(\eps^2). \end{align*}
Hence, again using \eqref{hodo}, and the free boundary condition for $\bar U_s$, we get 
$$1=| \nabla \bar U_s(x-\psi_s\nu_x)|^2=1+2  (\p_\nu v_s -v_s \, (\p_{\nu\nu}U_0)) +O(\eps^2),$$ 
which gives the second condition in \eqref{le}.
Rescaling back we get  that $v$ satisfies \eqref{veq}, as desired.
Moreover, by Lemma \ref{asy},
$$O\left(\frac{v}{|x|}\right)= o(1), \quad \text{as $|x| \to \infty$}.$$\end{proof}

Lemma \ref{linearized1} remains valid if $v:= U_0 - \underline U$, after extending $\underline U$ analytically in  $\Omega_0 \cap B_R^c$ (for $R$ large enough), and moreover $v > 0$ in this region. Indeed, in the proof above, \eqref{under} holds for $\underline U_s$. Thus $\underline U_s$ can be extended analytically in $\Omega_0 \cap (B_2 \setminus \bar B_{1})$, and all the  estimates following \eqref{under} remain valid in $\Omega_0$.

\begin{rem}\label{expfb}Notice that, from \eqref{expexp}, we deduce the expansion ($R$ large)
$$\psi(x) = v(x)+ O\left(\frac{v^2}{|x|}\right), \quad x\in F(U_0) \cap B_R^c.$$
\end{rem}

While $v$ solves the perturbed linearized equation above in a subset of $\Omega_0$, at times our analysis leads to estimates that must be extended outside of $\Omega_0$. For that purpose we use the following remark. Here $A_l$ is the annulus defined in the previous Lemma.

  \begin{rem}\label{rem} If $u_1, u_2$ are critical points to $J$ in $A_0$, with $u_i \geq U_0$ and $$\|u_i - U_0\|_{L^\infty} \leq \eps, \quad u_1 \geq u_2+c\eps \quad \text{in $\bar\Omega_0 \cap A_0$},$$ for some $c>0$ and $\eps>0$ small depending on $c$, then $$u_1\geq u_2 \quad \text{in $A_{1/5}.$}$$ Indeed, if we argue as for $\bar U_s$ above, we obtain that if $\psi_i$ is the associated function defining the free boundary of $u_i$ as in \eqref{psi}, then for $x \in A_{1/5} \cap F(U_0)$, 
  $$u_i(x-t\nu)= u_i(x) - t + O(\eps^2), \quad t\in [0, \psi_i(x)].$$
  Hence, for all $t$'s for which this expansion holds and $\eps$ small,
  $$u_2(x-t\nu) \leq u_1(x) - c \eps - t + O(\eps^2) \leq u_1(x - t\nu) - c \eps + O(\eps^2) < u_1(x -t\nu),$$
which gives the desired result.
  
  \end{rem}

 \subsection{Regularity.} This subsection is devoted to the proof of parts $(ii)-(vi)$ of Theorem \ref{main}.  It relies on the following asymptotic expansion, which will be derived in the next section. Here we use $V_{\gamma_\pm}$, which have been defined in the introduction (see \eqref{defv},\eqref{defv2}). 
  
 \begin{prop}\label{vas} Let $v>0$ be a classical solution to \eqref{veq}. If $\gamma_- \neq \gamma_+$, there exist $\alpha'>0$ small universal and $R_0>0$ large universal, such that in $B_{R_0}^c \cap \Omega_0$, either \begin{equation}\label{1}v(x)=a_- V_{\gamma_-}(x)+ O(|x|^{-\gamma_- - \alpha'}), \quad a_->0,\end{equation} or 
\begin{equation}\label{2}v(x)=a_+ V_{\gamma_+}(x)+ O(|x|^{-\gamma_+ - \alpha'}), \quad a_+>0.\end{equation} If $\gamma_-=\gamma_+$, then  in $B_{R_0}^c \cap \Omega_0$
  \begin{equation}\label{3}v(x)= \bar a V_{\gamma_-}(x)+ \bar b V_{\gamma_+}+ O(|x|^{-\gamma_0 - \alpha'}), \quad
   \text{$\bar a \geq 0, \bar b \in \R$,  $\max\{\bar a, \bar b\}>0.$}
   \end{equation}
\end{prop}

With this proposition at hands, we can now obtain $(ii)$ in Theorem \ref{main}, that is the following result.

\begin{thm}\label{anal}$F(\bar U)$ is analytic.\end{thm}

\begin{proof} 
Towards the proof of analyticity, we wish to obtain the following claim.

\medskip

\noindent \textbf{Claim 1.}  There is a universal (large)  $R_0>0$ such that
$$\nabla \bar U(x) \cdot x - \bar U(x)< 0, \quad \text{in $\overline{\{\bar U>0\}}\cap B_{R_0}^c$}.$$ 
 
 In order to prove Claim 1, we set $v:= \bar U - U_0$. Recall that, in view of Lemma \ref{linearized1}, $v>0$ satisfies \eqref{veq}, and hence the asymptotic in Proposition \ref{vas}.
Notice that since $U_0$ is homogeneous of degree 1, 

$$\nabla \bar U \cdot x - \bar U= \nabla v \cdot x - v, \quad \text{in $\Omega_0$.}$$

On the other hand, by Proposition \ref{vas}, \begin{equation}\label{nablaex} v = Z_{\gamma_0}+ O(|x|^{-\gamma_0 - \alpha'}) \quad \text{in $B_{R_0}^c \cap \Omega_0$},\end{equation} with $Z_{\gamma_0}$ given by 
$$Z_{\gamma_0}=a V_{\gamma_0} \quad \text{if $\gamma_- \neq \gamma_+$}, \quad a>0,$$
and $\gamma_0$ either $\gamma_-$ or $\gamma_+$,
or  
$$Z_{\gamma_0}=a V_{\gamma_-}+ b V_{\gamma_+} \quad \text{if $\gamma_- = \gamma_+=\gamma_0$}, \quad 
\text{$a \geq 0,  b \in \R$,  $\max\{ a,  b\}>0$}. $$ 

Now, let $v_s(x)=v(sx)/s$ and similarly $(Z_{\gamma_0})_s(x)=Z_{\gamma_0}(sx)/s$. Then, for $s$ large, according to \eqref{nablaex},
$$v_s(x)= (Z_{\gamma_0})_s(x) + s^{-\gamma_0-\alpha'-1}O(1), \quad \text{in  $(B_2 \setminus \bar B_1) \cap \Omega_0,$ }$$
and from the formula for $Z_{\gamma_0}$ (see also \eqref{defv}-\eqref{defv2}),
$$\eps= \eps(s):= \begin{cases}\|(Z_{\gamma_0})_s\|_\infty \sim s^{-\gamma_0-1}, \quad \text{if $\gamma_- \neq \gamma_+$,}\\
 \|(Z_{\gamma_0})_s\|_\infty \sim s^{-\gamma_0-1}\ln s, \quad \text{if $\gamma_- = \gamma_+$}.\end{cases}$$
From this we deduce that, for $ \sigma:= \frac{\alpha'}{2(\gamma_0+1)}, $
\begin{equation}\label{000}v_s(x) = (Z_{\gamma_0})_s(x) + O(\eps^{1+\sigma}), \quad \text{in $(B_2 \setminus \bar B_1)\cap \Omega_0$},\end{equation} 
and hence for   universally   large  $s$, we have 
$\|v_s\|_\infty \leq C \eps$ in the annulus $(B_2 \setminus \bar B_1) \cap \Omega_0.$
Next, in view of the first inequality in \eqref{hodo} we have (with the notation for annuli in the proof of Lemma \ref{linearized1}),
$$\left\|\frac{v_s(x) - (Z_{\gamma_0})_s(x)}{\eps}\right\|_{2,\alpha} \leq C, \quad \text{in $A_{2/5}\cap \Omega_0,$}$$
 while from the expansion above,
$$\left\|\frac{v_s(x) - (Z_{\gamma_0})_s(x)}{\eps}\right\|_\infty \leq C\eps^\sigma, \quad \text{in $A_{2/5}\cap \Omega_0$}.$$
We can interpolate to obtain (because $\p \Omega_0 \setminus \{ 0 \}$ is smooth),
$$\left\|\frac{v_s(x) - (Z_{\gamma_0})_s(x)}{\eps}\right\|_{0,1} \leq C\eps^{\sigma'},\quad \text{in $A_{7/15}\cap \Omega_0.$}$$
Combining \eqref{000} with the estimate above, we now compute (for $s$ large)
$$
\nabla v_s \cdot x - v_s = \nabla (Z_{\gamma_0})_s \cdot x - (Z_{\gamma_0})_s+ O(\eps^{1+\sigma'})
$$  
$$\leq -C\eps (\gamma_0+1)(1+o(1)) + O(\eps^{1+\sigma'}) \leq -c \eps<0 \quad \text{in $A_{7/15} \cap \Omega_0,$}
$$ where the first inequality follows from the formula for $Z_{\gamma_0}$ (see also \eqref{defv}-\eqref{defv2}).
Unraveling the scaling in the inequality above, we obtain Claim 1 in  $B_{R_0}^c \cap \Omega_0$. However, we need to extend this inequality to $\overline{\{\bar U_s>0\}}$. Let us take a point $x_0 \in F(U_0)$ and argue as in Lemma \ref{linearized1}, 
 that is, in view of the $C^{2,\alpha}$ estimate on $\bar U_s$ and the fact that 
$$\bar U_s(x_0) = O(\eps), \quad \p_\nu\bar U_s(x_0)= 1 + O(\eps),$$ 
we conclude that $F(\bar U_s)$ is included in an $\eps$-neighborhood of $F(U_0)$. On the other hand, using that the $C^{2,\alpha}$ norm of $v_s$ is controlled by its $L^\infty$ norm, and that by homogeneity $\nabla U_0 \cdot x - U_0 \equiv 0$ in $A_{7/15} \cap \bar \Omega_0$, we get 
$$(\nabla \bar U_s \cdot x - \bar U_s )(x_0)=(\nabla v_s \cdot x - v_s )(x_0) \leq -C\eps,$$ 
$$ \p_\nu(\nabla \bar U_s \cdot x - \bar U_s)(x_0)=\p_\nu(\nabla v_s \cdot x - v_s)(x_0)=O(\eps),$$
from which we obtain (using the fact that the $C^{2,\alpha}$ norm of $\bar U_s$ is universally bounded) that for $s$ large universal,$$\nabla \bar U_s \cdot x - \bar U_s <0 \quad \text{in $A_{7/15}\cap \overline{\{\bar U_s>0\}}$,}$$ as desired.

Claim 1 implies,
 $$\frac{d}{dt}\bar U_t (x) < -\frac{c}{t^2} \quad \text{for $tx \in \overline{\{\bar U>0\}} \cap (B_{4R_0} \setminus B_{R_0})$},
 \qquad (c=c(R_0)).
 $$ Given $x \in \{\bar U > 0\}$ let $\bar t=\bar t(x)$ be the smallest value for which the ray $tx \in \overline{\{\bar U>0\}} \cap (B_{4R_0} \setminus B_{R_0})$ for all $t < \bar t$, then 
\begin{equation}\label{iii}\bar U_{\bar t}(x) - \bar U(x)= -\int_{\bar t}^1 \frac{d}{dt}\bar U_t (x)\; dt \geq c (1-\bar t).\end{equation} 
From this we deduce that $\bar U_{\bar t}(x)>0,$ hence $\bar U_t (x)$ is strictly decreasing in $t \in (1/2,1]$ when $x \in \overline{\{\bar U > 0\}} \cap (B_{3R_0}\setminus B_{2R_0})$.
Now, for $\delta$ small and  $x \in \overline{\{\bar U > 0\}} \cap (B_{3R_0}\setminus B_{2R_0})$, the same computation as in \eqref{iii} gives that 
$$\bar U_{1-\delta}(x) - \bar U(x) \geq c \delta. $$
On the other hand, for any unit vector $\tau$, since $\bar U_{1-\delta}$ is Lipschitz,  with Lipschitz constant in the annulus independent of $\delta$, it follows from the inequality above that for $x \in \overline{\{\bar U > 0\}} \cap (B_{3R_0}\setminus B_{2R_0}),$
$$\bar U_{1-\delta} (x+ c_1\delta \tau) \geq \bar U_{1-\delta}(x) - Lc_1\delta \geq \bar U(x),$$
as long as $c_1$ (being universal) is small enough.
 Hence,
$$\bar U_{1-\delta} (x+ c_1\delta \tau) \geq \bar U(x), \quad 2R_0 \leq |x| \leq 3R_0,$$ and
by the comparison principle Proposition \ref{comp1}, we conclude that this inequality holds for all $x$'s. Therefore
 $F(\bar U)$ is a Lipschitz radial graph, and by standard  arguments it is also locally Lipschitz. By the classical regularity theory for one phase free boundaries \cite{AC, C1, KN}, we conclude that $F(\bar U)$ is analytic. 
\end{proof}

We can now deduce the proof of parts $(iii)-(vi)$ in Theorem \ref{main}.

\medskip

\noindent\textit{Proof of Theorem \ref{main} (iii)-(vi).} Claim 1 in the proof of Theorem \ref{anal}  and Proposition \ref{vas}, provide the statements in part $(iii)-(iv)$ of Theorem \ref{main}. The statement in $(v)$ follows immediately from $(iv).$ 
We are left with the proof of $(vi)$, from which we also deduce the universality of the coefficient in the expansion in $(iii)$. Assume $V$ is not identically equal to $U_0$. Then the same arguments in Lemma \ref{asy}, Lemma \ref{linearized1} and Proposition \ref{vas} can be applied to $V$ giving that $V$ is asymptotic to $U_0$, and 
 \begin{equation} V- U_0 = V_{\gamma_1}(d+ o(1)) \quad \text{in $B_{R_0}^c \cap \Omega_0$},\end{equation} with $\gamma_1$ either $\gamma_-$ or $\gamma_+$, and $d>0$.
 
On the other hand, 
\begin{equation} \bar U- U_0 = V_{\gamma_0}(a+ o(1)) \quad \text{in $B_{R_0}^c \cap \Omega_0$},\end{equation}
with  $\gamma_0$ either $\gamma_-$ or $\gamma_+$, and $a>0$.

If $\bar U$ and $V$ both have expansions in terms of $V_{\gamma_-}$ (resp. $V_{\gamma_+}$), then we can argue as follows. Define $t_0$ by, $$a t_0^{-\gamma_- -1}= d.$$ 
Using the expansions above for a given $t>t_0$, we conclude that for $|x|$ large, $x \in \Omega_0,$ 
$$\frac{V- \bar U_t}{V_{\gamma_-}} = (d-at^ {-\gamma_--1} + o(1)), \quad d-at^{-\gamma_--1}>0.$$
Thus, for any $t>t_0$, there exists $R_t$ large such that in an annulus $(B_{2R_t} \setminus B_{R_t}) \cap \Omega_0$, 
$$V \geq c(t) \tilde\eps + \bar U_t, \quad \tilde\eps = \|V_{\gamma_-}\|_\infty,$$ where we have used that $V_{\gamma_-} \sim \|V_{\gamma_-}\|_\infty$ in $(B_{2R_t} \setminus B_{R_t}) \cap \Omega_0$.
By a rescaled version of Remark \ref{rem} we get $V \geq \bar U_t$ in the annulus  $B_{2R_t} \setminus B_{R_t}$, and the comparison Proposition \ref{comp1} gives, 
$$V \geq \bar U_t, \quad \forall t>t_0, \quad \text{in $\R^n$}. $$ Similarly, 
$$V \leq \bar U_t, \quad \forall t<t_0, \quad \text{in $\R^n$},$$
and the claim follows letting $t \to t_0.$ If the expansions are different, say $\bar U$ and  $V_{\gamma_-}$ 
have a expansions in terms of $V$ respectively  $V_{\gamma_+}$, then arguing as above we conclude that $U_t \geq V$ for all $t$'s, and by letting $t\to \infty$ we obtain $U_0 \geq V$, hence $U_0 \equiv V$ a contradiction. 
\qed

\section{The perturbed linearized problem}

This section is intended for the proof of Proposition \ref{vas}. For convenience, we recall some notation from the introduction and refer the reader to \cite{JS} for further details on the discussion below. Consider  the problem, \begin{equation}\label{w4}\begin{cases}
\Delta w=0 \quad \text{in $\Omega_0$}\\
\p_\nu w +Hw=0 \quad \text{on $\p \Omega_0 \setminus \{0\}$},\\
\end{cases}\end{equation} with $\nu$ the interior unit normal to $\p \Omega_0,$ and $-(\p_{\nu\nu}U_0)=H>0$ the mean curvature of $\p \Omega_0$ oriented towards the complement of the connected set $\Omega_0.$ 
Let  $w:= f(r) \bar v(\theta)$, with $f\ge 0$ a radial function, $r:=|x|$, and $\bar v$ the corresponding first eigenfunction of the Laplacian on $\mathbb S^{n-1} \cap \Omega_0$, i.e., 
$$\Delta_{\mathbb S^{n-1}} \bar v= \lambda \bar v, \quad \text{on $\mathbb S^{n-1} \cap \Omega_0,$}$$   satisfying the Neumann condition, 
\begin{equation}\label{ei}
\p_\nu\bar v + H \bar v=0, \quad \text{on $\p \Omega_0 \cap \mathbb S^{n-1}.$}
\end{equation} 
Then $\bar v>0$ on $\mathbb S^{n-1} \cap \bar \Omega_0$ and $\lambda>0$. We compute, 
$$
\Delta w= \bar v \Delta f + 2 \nabla \bar v \cdot \nabla f + f \Delta \bar v = \bar v\left (f'' + (n-1) \frac {f'}{r}+\lambda \frac{f}{r^2}\right),
$$
thus, for $f=r^{-\gamma}$, we obtain that $w$ solves \eqref{w4}
as long as $\gamma=\gamma_\pm$
 satisfy $\gamma^2 - (n-2)\gamma + \lambda=0$. The stability of $U_0$ is equivalent to the fact that this quadratic equation must have real roots i.e. $(n-2)^2- 4\lambda \geq 0$. Moreover, $\lambda>0,$ thus $\gamma=\gamma_\pm \in \R, \gamma_+\geq \gamma_->0$.

  If $\gamma_- \neq \gamma_+$ we call
$$V_{\gamma_\pm}(x) := |x|^{-\gamma_{\pm}} \bar v,$$
while by abuse of notation, if $\gamma_-=\gamma_+$ and we call $\gamma_0$ this common value, we set 
\begin{equation}\label{equal}V_{\gamma_-}(x):= |x|^{-\gamma_0} (\ln |x|+1) \; \bar v, \quad V_{\gamma_+}(x):= |x|^{-\gamma_0} \; \bar v.\end{equation}

Next, we build the following special family of functions, which will play an essential role in the proof of  Theorem \ref{main}. Call $u_0:= |x|^{-1}U_0,$ 
and define for real numbers $\gamma$, $ \beta$, $ A =A(\gamma)$ 
\begin{equation}\label{family}
W_{\gamma}^\beta(x)= \begin{cases}|x|^{-\gamma}(A  \bar v + u_0)\quad \text{$A>0$ if $0< \gamma < \gamma_-$},
\\
|x|^{-\gamma}(A  \bar v + u_0) \quad \text{$A<0$ if $\gamma_-< \gamma < \gamma_+$},
\\
|x|^{-\gamma}(A  \bar v + u_0) + \beta V_{\gamma_+} \quad \text{$A>0$ if $\gamma > \gamma_+.$}
\end{cases}\end{equation}
In what follows, when $\gamma < \gamma_+$, we drop the dependence on $\beta$. Notice that our definition also includes the case $\gamma_-=\gamma_+.$
Then, using the  above formula,
$$\Delta W^\beta_{\gamma}=|x|^{-\gamma-2}\{A[\gamma(\gamma-n+2)+\lambda]  \bar v+ u_0[\gamma(\gamma-n+2)+(1-n)]\}$$
and for $|A|$ large enough,
$$\Delta W^\beta_\gamma \geq 0, \quad \text{in $\Omega_0$}.$$ Moreover, using that $\bar v$ solves \eqref{ei}, while $u_0=0, (u_0)_\nu=1/r$ on $\p \Omega_0 \setminus \{0\}$, we get
\begin{equation}\label{bd}\p_\nu W^\beta_{\gamma} +HW^\beta_{\gamma}= |x|^{-\gamma-1} \quad \text{on $\p \Omega_0 \setminus \{0\}$}.\end{equation}
We remark that $W_\gamma > 0$ in $\Omega_0$ when $\gamma< \gamma_-$, and choosing $|A|$ large, $W_\gamma <0$  in $\Omega_0$ when $\gamma_-<\gamma< \gamma_+$. The sign of $W_\gamma^\beta$ depends on the choice of $\beta.$

Having introduced this family of functions, we can now provide the proof of Proposition \ref{vas}.
%
%
%
The proof relies on a comparison principle for solutions to \eqref{veq} in an annulus $(B_R \setminus \bar B_1) \cap \Omega_0$, which we prove in the next lemma. Consider the problem,
\begin{equation}\label{veqsection4general}\begin{cases}
\Delta V=0 \quad \text{in $D$}, \\
\p_\nu V = h(x)V\quad \text{on $T \subset \p D$},\\
\end{cases}\end{equation} with $D$ a Lipschitz domain in $\R^n$, $T$ a smooth open subset of $\p D$, and $\nu$ the inward unit normal to $T$. 

\begin{lem}\label{COMP} If there exists a classical supersolution to \eqref{veqsection4general} (continuous up to $\p D$) such that $V>0$ on $\bar D$, then the comparison principle holds i.e. if $z,w$ are respectively a classical supersolution and a classical subsolution to  \eqref{veqsection4general}, with $z \geq v$ on $\p D \setminus T$, then $z \geq v$ in $D.$
\end{lem}
\begin{proof} Let $w:= v-z$, and set  $M:= \max_{\bar D}{\frac w V} = \frac{w}{V}(x_0)$ for some $x_0 \in \bar D$. If $x_0 \in \p D \setminus T$, then by our assumptions $M \leq 1,$
 which implies our claim. Similarly, if $x_0 \in D$, then $MV - w$ attains a minimum at $x_0$ and by the maximum principle $MV \equiv w$, which contradicts out assumptions. Finally, $x_0$ cannot occur on $T$. Indeed, again $M V -w \geq  0$ attains a minimum at $x_0$ and  by Hopf Lemma, $\p_\nu (M V -w)(x_0)>0$. On the other hand,
$$0 < M \p_{\nu} V(x_0) -\p_{\nu}w(x_0) \leq h(x_0)(MV-w)(x_0)= 0,$$
 a contradiction. \end{proof}

Besides the above comparison principle we also use the Harnack inequality for inhomogeneous Neumann problems, see for example \cite{L}. The precise statement is provided in the Appendix.

\smallskip

We are now ready to prove Proposition \ref{vas}.

\medskip

\noindent\textit{Proof of Proposition \ref{vas}}. The proof is divided into 5 steps. For simplicity, in each step we consider first the case when $\gamma_- \neq \gamma_+$, and then point out the modifications needed for the case $\gamma_-=\gamma_+$. 

Since $U_0$ is homogeneous of degree $1$, after a dilation $x \to \rho x$, $\rho=\rho(\delta)$ large, we may assume that
\begin{equation}\label{veqdilat}\begin{cases}
\Delta v=0 \quad \text{in $\Omega_0 \cap B_{1}^c$}\\
\p_\nu v +H(1+ o(1))v=0 \quad \text{on $\p \Omega_0 \cap B_{1}^c$},\\
\end{cases}\end{equation}
with
\begin{equation}\label{o1}|o(1)| \leq \delta,\end{equation}
for $\delta>0$ to be made precise later. By abuse of notation, in what follows all dilations of $v$ will still be denoted by $v$.

\

\textbf{Step 1: Decay at infinity.} Let $v$ satisfy \eqref{veqdilat}-\eqref{o1}. Given $0<\gamma< \gamma_-$, there exists $M= M(\gamma, v)>0$ large such that
 $$v \leq M |x|^{-\gamma} \quad \text{on $\Omega_0 \cap B_1^c.$}$$

For this step, we do no need to distinguish whether $\gamma_-$ and $\gamma_+$ coincide or not.
Indeed, let $\gamma$ be given and the constants below possibly depend on $\gamma$. To prove the desired bound, we use the function $W_\gamma$ defined in \eqref{family} (we drop the dependence on $\beta$ in the regime $\gamma<\gamma_-$), and we construct a subsolution $W=- CV_{\gamma_-} + c W_{\gamma}$  to \eqref{veqdilat}, such that for large $C$ and  small $c$, \begin{equation}\label{b1}W \leq R^{-\gamma} \quad \text{on $\p B_R \cap \Omega_0$, $R$ large}, \quad W \leq 0 \quad \text{on $\p B_1 \cap \Omega_0$}.\end{equation} This is possible because $\gamma < \gamma_-,$ hence $W_{\gamma}>0$ is the leading term. For the boundary condition to yield a subsolution we need,
$$\p_\nu W +H(1+o(1))W\geq 0 \quad \text{on $\p \Omega_0 \cap B_{1}^c$},$$ which in view of \eqref{bd} and \eqref{o1} holds as long as (recall that $V_{\gamma_-}$ solves \eqref{w4})
$$-CH\delta V_{\gamma_-} + c |x|^{-\gamma-1} - c\delta H W_\gamma \geq 0, \quad \text{on $\p \Omega_0 \cap B_{1}^c.$}$$
This can be achieved by choosing $\delta$ small.

Moreover, again since $W_\gamma$ is the leading term, we can pick $\bar R < R$ large, such that
\begin{equation}\label{outb2}\bar R^{-\gamma_-} \leq W \quad \text{on $\p B_{\bar R} \cap \Omega_0$}.\end{equation}

Now, for $M$ large to be specified later, let us assume by contradiction that $v(\bar x)  > M R^{-\gamma}$ at some point on 
$\partial B_R\cap  \Omega_0$. 
Then by the interior Harnack inequality and the Harnack inequality for the Neumann problem \eqref{veq} (see Theorem \ref{hn}),
 we have that $v  \geq \bar c M R^{-\gamma}$ on $\p B_R \cap \overline \Omega_0$, with $\bar c>0$ universal. Thus, using \eqref{b1},  and the fact that $v \geq 0,$ we have that 
$$v \geq \bar cM W \quad \text{on $(\p B_{R} \cup \p B_{1}) \cap \Omega_0$}.$$ By Lemma \ref{COMP}, which can be applied because $v$ itself is a solution to \eqref{veqdilat} which is strictly positive up to the boundary, we get 

$$v \geq \bar cM W \quad \text{in $(B_{R} \setminus B_{1})\cap \Omega_0$}.$$ In particular, by \eqref{outb2} above,
$$v \geq c(\bar R)M \quad \text{on $\p B_{\bar R} \cap \Omega_0,$}$$ hence we reach a contradiction if $M$ is large enough, and Step 1 is proved. 

\smallskip

Now, choosing $\gamma=\gamma_-/2$, for a universal $\alpha_0=\gamma_-/2+9/10$ we have by Step 1 that, $$\frac{v(x)}{|x|}= o(|x|^{-\alpha_0}), \quad |x| \to \infty$$ hence, again after a dilation, we can assume that $v$ satisfies:
\begin{equation}\label{pertalpha}\begin{cases}
\Delta v=0 \quad \text{in $\Omega_0 \cap B_1^c,$}\\
\p_\nu v +H (1+ o(|x|^{-\alpha_0}))v=0 \quad \text{on $\p \Omega_0 \cap B_1^c$},\\
\end{cases}\end{equation}
with 
\begin{equation}\label{o2}|o(|x|^{-\alpha_0})| \leq \delta |x|^{-\alpha_0}\end{equation}
and $\delta$ small universal to be chosen later.


\smallskip

\textbf{Step 2: Improved Decay at infinity.} Let $v$ satisfy \eqref{pertalpha}-\eqref{o2}. There exists $M_->0$ large, such that $$\frac{v}{V_{\gamma_-}} \leq M_- \quad \text{on $ \Omega_0 \cap B_1^c.$}$$

The argument follows the lines of Step 1, but we need to distinguish whether $\gamma_-$ and $\gamma_+$ coincide or not. If the roots are distinct, fix $\gamma \in (\gamma_-, \gamma_+)$. We construct a subsolution $W=V_{\gamma_-} + W_{\gamma}$ to \eqref{pertalpha}, such that ($R=R(\gamma)$ large), 
\begin{equation}\label{outb3}
W \leq \frac 3 2 V_{\gamma_-}\quad \text{on $\p B_R \cap \Omega_0$}, \quad W \leq 0 \quad \text{on $\p B_1 \cap \Omega_0$}.
\end{equation}
The desired inequalities can be achieved as $W/V_{\gamma_-} \to 1$ as $r \to \infty$ and by choosing $|A|$ possibly larger. Moreover, for the same reason, we can pick 
 a large $\bar R <R$ such that 
\begin{equation}\label{outb4}\frac 1 2 V_{\gamma_-} \leq W \quad \text{on $\p B_{\bar R} \cap \Omega_0$}.
\end{equation}
Now, for the boundary inequality to be satisfied, we need to choose (see \eqref{o2}), $$\gamma_- < \gamma \leq \gamma_- + \alpha_0,$$ and $\delta$ small enough, so that
\begin{equation}\label{b2} -\delta H|x|^{-\alpha_0} V_{\gamma_-} +  |x|^{-\gamma-1} - \delta H|x|^{-\alpha_0} W_\gamma \geq 0, \quad \text{on $\p \Omega_0 \cap B_{1}^c.$}\end{equation}

 The argument is now the same as in Step 1. Given $M$ large, to be specified later, let us assume by contradiction that $\frac{v}{V_{\gamma_-}}  \geq M$ at some point  at some point on  $\partial B_R\cap  \Omega_0$. 
Then by the Harnack inequality (interior and for a Neumann problem as in Theorem \ref{hn}),
we have that $\frac{v}{V_{\gamma_-}}  \geq c M$ on $\p B_R \cap \overline \Omega_0$. 
This combined with \eqref{outb3} and Lemma \ref{COMP} (again since $v>0$) gives that,
$$\frac{v}{V_{\gamma_-}} \geq \bar cM \frac{W}{V_{\gamma_-}} \quad \text{in $(B_{R} \setminus B_{1})\cap \Omega_0$}.$$
Hence, in view of \eqref{outb4}, $$v \geq c(\bar R)M \quad \text{on $\p B_{\bar R} \cap \Omega_0$},$$
 which is  a contradiction if $M$ is large enough. 
 
 In the case when $\gamma_- =\gamma_+$, we need to choose $W=V_{\gamma_-} + W_\gamma^\beta$ and let $\beta$ be negative and $|\beta|$ large, so that the second inequality in \eqref{outb3} holds (see \eqref{family} for the definition of $W_\gamma^\beta$).

\

\textbf{Step 3: Limit at infinity and expansion.} Let $v$ satisfy \eqref{pertalpha}-\eqref{o2}. Then, $$\lim_{|x|\to\infty} \frac{v}{V_{\gamma_-}} = a_-\geq0.$$
Moreover, if $\gamma_-\neq \gamma_+,$
\begin{equation}\label{sum1}v(x)=a_- V_{\gamma_-}(x)+ O(|x|^{-\gamma_- - \alpha'}), \quad \text{in $\Omega_0 \cap B_1^c$},\end{equation} while if $\gamma_-=\gamma_+$,
\begin{equation}\label{sum2}v(x)=a_- V_{\gamma_-}(x)+ O(V_{\gamma_+}), \quad \text{in $\Omega_0 \cap B_1^c$}.\end{equation}

We consider first that case when $\gamma_- \neq \gamma_+$. For $\rho \geq 1$, let $$a(\rho):= \sup \{a \geq 0 \ | \ v \geq a V_{\gamma_-} \ \text{in $B_\rho^c \cap \Omega_0$}\}.$$

This is an increasing function which in view of Step 2 is bounded by $M_-$. Thus, there exists
$$a_-:= \lim_{\rho \to \infty} a(\rho) \geq 0,$$ and by definition of $a(\rho)$,
\begin{equation}\label{liminf}\liminf_{|x|\to\infty} \frac{v}{V_{\gamma_-}} \geq a_-.\end{equation}
We wish to show that 
$$\limsup_{|x|\to\infty} \frac{v}{V_{\gamma_-}} \leq a_-,$$
from which our claim will follow. Assume by contradiction that for a small $\eta>0$ along a sequence of points $x_k \in \Omega_0$ with $\rho_k:=|x_k| \to \infty$,
\begin{equation}\label{bobo}\frac{v}{V_{\gamma_-}}(x_k) \geq a_- +\eta,\end{equation}
while in view of \eqref{liminf}, for $k$ large and $\eps \ll  \eta$,
\begin{equation}\label{bo}\frac{v}{V_{\gamma_-}} \geq a_--\eps \quad \text{in $B_{\rho_k/2}^c \cap \Omega_0$}.\end{equation}
Thus, by the Harnack inequality (Theorem \ref{hn}) applied to 
 $v-(a_--\eps)V_{\gamma_-}$ we conclude that for $k$ large,
\begin{equation}\label{HC}\frac{v}{V_{\gamma_-}} \geq a_-+ c\eta, \quad \text{on $\p B_{\rho_k}$} \cap \Omega_0.\end{equation}
Indeed, let us call $w:= v-(a_--\eps)V_{\gamma_-}$ and consider for $\frac 1 2 < |x|< 2$ the functions $$w_k(x)= \rho_k^{\gamma_-} v(\rho_k x) -(a_--\eps) V_{\gamma_-}(x). $$ In view of \eqref{bo}, $w_k \geq 0,$ and by \eqref{bobo}, $w_k(\bar x) \geq  c' \eta$ at some point $\bar x \in \p B_1 \cap \Omega_0$. On the other hand, $v_k(x)= \rho_k^{\gamma_-} v(\rho_k x)$
satisfies
$$\begin{cases}
\Delta v_k=0 \quad \text{in $\Omega_0 \cap (B_2 \setminus \bar B_{1/2}),$}\\
\p_\nu v_k +H(1+ o(1))v_k=0 \quad \text{on $\p \Omega_0 \cap (B_2\setminus \bar B_{1/2})$},\\
\end{cases}$$
and in view of Step 2, it is uniformly bounded. Thus by standard regularity estimates, the $v_k$'s are uniformly H\"older continuous in $(B_{15/8} \setminus B_{5/8}) \cap \Omega_0$ and, up to extracting a subsequence, $w_k$ converges uniformly to a nonnegative limiting function $\bar w$ which solves the unperturbed problem \eqref{w4} in $(B_{7/4} \setminus \bar B_{3/4}) \cap \Omega_0$ and $\bar w (\bar x) \geq c'\eta$ at some point on $\p B_1 \cap \Omega_0$. By the Harnack inequality and the uniform convergence we get,
$w_k \geq c \eta$ on $\p B_1 \cap \Omega_0$ for $c$ universal, and unraveling the scaling this gives the desired claim \eqref{HC}.

As in Step 2, we now build a subsolution
$$W= (a_- + \frac c 2 \eta)V_{\gamma_-} + W_\gamma, \quad \gamma_-<  \gamma \leq \gamma_-+\alpha_0, \quad \gamma < \gamma_+,$$ 
such that for $k\geq k_0$ large
 $$W \leq (a_-+c\eta)V_{\gamma_-} \quad \text{on $\p B_{\rho_k} \cap \Omega_0.$}$$
Thus by Lemma \ref{COMP}, for $k$ large,
$$v \geq W \quad \text{in $(B_{\rho_k} \setminus B_{\rho_{k_0}}) \cap \Omega_0$},$$ hence outside a very large ball 
$$\frac{v}{V_{\gamma_-}} \geq a_- + \frac c 4 \eta, \quad \text{in $B_{R}^c \cap \Omega_0$}, $$
contradicting the definition of $a_-$.
Finally, in order to obtain the expansion, we use a similar argument as above to trap $v$ in any annulus $B_R \setminus B_1$ with $R$ large, between a subsolution and a supersolution, respectively, 
\begin{equation}\label{ss}(a_- - \eps)V_{\gamma_-} + CW_\gamma, \quad (a_- + \eps)V_{\gamma_-} - C W_\gamma,\end{equation}
by choosing $C$ large enough. Letting $\eps$ go to zero we get the desired expansion with $\alpha'=\gamma$. 
Rescaling back, we obtain the desired expansion. 

There  are two modifications needed in the case $\gamma_-=\gamma_+$. The first is the definition of $w_k$ and $v_k$, that is
$$w_k(x)= \frac{1}{\ln \rho_k + 1}\rho_k^{\gamma_-} v(\rho_k x) -(a_--\eps)[ V_{\gamma_+}(x) + \frac{1}{\ln \rho_k + 1}V_{\gamma_-}(x)], $$ 
and
$$v_k(x)=\frac{1}{\ln \rho_k + 1} \rho_k^{\gamma_-} v(\rho_k x).$$

The second one is in the construction of the subsolution/supersolution in the last step of the previous argument. Indeed the subsolution and the supersolution in \eqref{ss} must be replaced by
\begin{equation}\label{ss2}(a_- - \eps)V_{\gamma_-} + CW^{\beta}_\gamma, \quad (a_- + \eps)V_{\gamma_-} - C W^{\beta}_\gamma,\end{equation}
with $\beta<0$ and $|\beta|$ large.

\

\textbf{Step 4: The case $\gamma_- \neq \gamma_+$. (Comparison with $V_{\gamma_+}$}.)
 Let $v$ satisfy \eqref{pertalpha}-\eqref{o2} and assume that $a_-=0$ (from Step 3) and
 $$\frac{v}{V_{\gamma_-}} = O(|x|^{-\alpha'}), \quad \text{in $B_1^c$}, \quad \gamma_- < \alpha' \leq \gamma_-+\alpha_0, \quad \alpha'< \gamma_+.$$
There exists $M_+, R>0$ large, such that $$\frac{v}{V_{\gamma_+}} \leq M_+ \quad \text{on $B_{R}^c \cap \Omega_0.$}$$

Since $a_-=0$, we can repeat the same arguments as above with $\alpha_0$ replaced by $\alpha_1=\alpha'+\gamma_-+9/10 > \alpha_0$, that is a dilation of $v$ satisfies \eqref{pertalpha} with $\alpha_0$ replaced by $\alpha_1$. If  $ \gamma_-+\alpha_1 < \gamma_+$, we  use Lemma \ref{COMP} and compare $v$ with the supersolution (choose $\bar \gamma \leq \gamma_-+\alpha_1$ according to a similar computation as in \eqref{b2})
$$Z= \eps V_{\gamma_-} - C W_{\bar \gamma}.$$
We check that  $$v \leq Z \quad \text{on $\p B_{1}\cap \Omega_0$, \quad for $C$ large}$$ and $$v \leq Z \quad \text{on $\p B_R \cap \Omega_0$, \quad for $R$ large.}$$ The first inequality holds as $W_{\bar\gamma} < - \bar C<0.$ The second inequality follows as $a_-=0$ and $\frac{W_{\bar \gamma}}{V_{\gamma_-}} \to 0$ as $|x| \to \infty.$
As $\eps \to 0$ we get that $$v \leq -CW_{\bar\gamma} \quad \text{in $B_{1}^c \cap \Omega_0$}.$$ We repeat the same argument with $\alpha_2= \bar \gamma +9/10 > \alpha_1$, and we continue till we reach the first $l\geq 1$ for which $\gamma_-+\alpha_l > \gamma_+.$ Then, via Lemma \ref{COMP}, we compare $v$ with the supersolution ($\gamma_+< \bar \gamma \leq \gamma_-+\alpha_l $)
$$Z= \eps V_{\gamma_-} - C W^\beta_{\bar \gamma},$$ (again see \eqref{family} for the definition of $W_\gamma^\beta$) and choose $\beta < 0$  and  $|\beta |$ large enough to guarantee that $v \leq Z$ on $\p B_{1}\cap \Omega_0$ (a similar computation as in \eqref{b2} can be carried out again).
Therefore we obtain that $$v \leq -CW^\beta_{\bar\gamma} \quad \text{in $B_{1}^c \cap \Omega_0$},$$ and the claim follows from the definition of $W_{\bar\gamma}^\beta.$

\

\textbf{Step 5: Separation from 0 and Expansion.}
Let $v$ solve \eqref{pertalpha}-\eqref{o2} and whenever $\gamma_- \neq \gamma_+$
 replace $\alpha_0$ by $\alpha_l$ from Step 4. Then,
\begin{equation}\label{PP}v \geq \bar c V_{\gamma_+}, \quad \text{in $B_{2}^c \cap \Omega_0$, \quad $\bar c>0$}.\end{equation}
To see this, we  let $a= v(\bar x)>0$ for some $\bar x \in \p B_{2} \cap \Omega_0$. By the Harnack inequality in $(B_3\setminus \bar B_1) \cap \Omega_0$, $v \sim a$ on $\p B_{2} \cap \Omega_0.$  Define (the subsolution) $W:= -\eps V_{\gamma_-} + W^1_{\gamma}$ (with $\gamma>\gamma_+$),\footnote{In view of the iteration argument at the end of the previous step, this also works  in the case of distinct roots.} such that $W \leq C$ on $\p B_{2} \cap \Omega_0$ and $W \leq 0$ for $R$ large. As usual, we are using the function $W^\beta_\gamma$ defined in \eqref{family}. Then $v \geq ca W$ in $(B_R \setminus B_2) \cap \Omega_0$ and by letting $\eps \to 0$ and using the definition of $W^1_\gamma$, we get the desired statement.

To conclude the proof of the expansions \eqref{2}-\eqref{3}, we prove the following improvement of oscillation. First we consider the case of distinct roots.

\smallskip

\noindent {\it Improvement of oscillation.}
There exist positive sequences $\{a_k\}$ increasing and $\{b_k\}$ decreasing with $$b_{k+1} - a_{k+1} = (1-c)(b_k - a_k)$$ such that 
\begin{equation}\label{basis}
a_k \leq \frac{v}{V_{\gamma_+}} \leq b_k, \quad \text{in $B_{2^k}^c \cap \Omega_0$}.
\end{equation} 
Let \eqref{basis} hold for some $k \geq 1$,   call $\rho_k:= 2^k$, and for $\mu>0$ small, write $b_k - a_k = \rho_k^{-\mu}$. Define,
$$w(x):= \rho_k^{\mu+\gamma_+}(v - a_kV_{\gamma_+})(\rho_k x), \quad x \in B_1^c \cap \Omega_0.$$ Then $w$ is harmonic in $B_1^c \cap \Omega_0$ and it satisfies the boundary condition
$$w_\nu + Hw= o(|x|^{-\alpha_l - 1 -\gamma_+}) \quad \text{on $\p \Omega_0 \cap B_1^c$},$$ with $|o(|x|^{-\alpha_l - 1 -\gamma_+})| \leq \delta |x|^{-\alpha_l - 1 -\gamma_+}$. Moreover, in view of \eqref{basis},
$$0 \leq \frac{w}{V_{\gamma_+}} \leq 1, \quad \text{in $\Omega_0 \cap B_1^c.$}$$ Assume that $w \geq \frac 1 2$ at some point $x \in \p B_{3/2} \cap \Omega_0$. Then by the interior Harnack inequality and the Harnack inequality for a Neumann problem (see Theorem \ref{hn})  $\frac{w}{V_{\gamma_+}} \geq c$ on $\p B_{3/2} \cap \Omega_0$. We build a subsolution $W$ (to the problem satisfied by $w$), such that \begin{equation}\label{ww}
W \leq c, \quad \text{on $\p B_{3/2} \cap \Omega_0$}, \quad W \leq 0 \quad \text{on $\p B_R \cap \Omega_0$, $R$ large}.\end{equation}
Set
$$W= c_1 V_{\gamma_+} + c_2W_\gamma^0 - \eps V_{\gamma_-}, \quad \gamma> \gamma_+.$$
 Then, the first inequality in \eqref{ww} is satisfied for $c_1, c_2$ small, while the second one follows as $W/{V_{\gamma_+}} \to -\infty$ as $|x| \to \infty.$ We need to verify that 
$$W_\nu + HW\geq  \delta (|x|^{-\alpha_l - 1 -\gamma_+}) \quad \text{on $\p \Omega_0 \cap B_{3/2}^c$},$$
or equivalently, 
$$c_2 |x|^{-\gamma-1} \geq \delta |x|^{-\alpha_l - 1 -\gamma_+} \quad \text{on $\p \Omega_0 \cap B_{3/2}^c$}.$$
This holds for $\delta$ small, as long as $\gamma< \alpha_l + \gamma_+.$ By  Lemma \ref{COMP}, which can be applied because $V_{\gamma_+}$ is a solution to \eqref{w4}, which is strictly positive on the closure of $D:= (B_R \setminus \bar B_{3/2}) \cap\Omega_0$, while $W-w$ is a subsolution to \eqref{w4} in $D$, non-positive on $(\p B_R \cup \p B_{3/2}) \cap \Omega_0,$ thus
$$w \geq W \quad \text{in $(B_R \setminus B_{3/2}) \cap \Omega_0$}$$ and by letting $\eps \to 0$ we deduce,
$$w \geq c_1 V_{\gamma_+} \quad \text{in $B^c_{2} \cap \Omega_0,$}$$ which after unraveling what $w$ is, gives the desired improvement. 

When $\gamma_- = \gamma_+$ the only modification is that, using the expansion in Step 3, \eqref{basis} will be satisfied by $v -a_- V_{\gamma_-}$, and \begin{equation}\label{last}w:= w(x):= \rho_k^{\mu+\gamma_+}(v -a_- V_{\gamma_-}- a_kV_{\gamma_+})(\rho_k x).\end{equation} The rest of the proof  works in the same way.

The proof of the proposition is now complete. Indeed, in the case $\gamma_- \neq \gamma_+$, \eqref{sum1} immediately gives the desired expansions \eqref{1} if $a_- \neq 0.$ Otherwise, by standard arguments, the improvement of oscillation in Step 5 leads to \eqref{2} (after rescaling back), with $\bar a>0$ in view of \eqref{PP}. If $\gamma_-=\gamma_+$, again the improvement of oscillation in Step 5 and \eqref{PP} give \eqref{3} (see \eqref{last}).

\qed

\section{Appendix}
In this Appendix, we provide a  $C^{2,\alpha}$ estimate for the difference of two nearby smooth solutions of \eqref{visc} in terms of the $L^\infty$ norm. This estimate is essential in our proof of Theorem \ref{main}, and it relies on the use of the hodograph transform. We refer to \cite{KN} for further details on this important tool. Here universal constants only depend on dimension.

\begin{prop}\label{nearby}Let $u_1, u_2$ be classical solutions to \eqref{visc} in $B_1$, with 
\begin{equation}\label{flat}(x_n- \eps_1)^+ \leq u_1\leq u_2 \leq (x_n + \eps_1)^+,\end{equation} 
for $\eps_1$ small universal. Then,  $v:=u_2-u_1$ satisfies \begin{equation}\label{hhodo}\|v\|_{C^{2,\alpha}(B_{1/2} \cap \overline{\{u_1>0\})}} \leq C \|v\|_{L^\infty(B_1)}, \quad \frac{v(x)}{v(y)} \leq C, \quad x,y \in B_{1/2} \cap \overline{\{u_1>0\}},\end{equation} with $C>0$ universal.
\end{prop}
\begin{proof}Let $k=1,2$. In view of the flatness assumption \eqref{flat}, if $\eps_1$ is chosen small (depending on $n$) such that the improvement of flatness Theorem in \cite{C2} holds, then $\|u_k\|_{C^3(B_{3/4})} \leq C,$ with $C$ universal. Thus, for $c>0$ universal, (again by \eqref{flat}),
$$1-c \leq \p_n u_k \leq  1+c, \quad  \text{in } B_{3/4}\cap \{u_k > 0\}.$$ 
Then, we can perform the partial hodograph transform  $y'=x', y_n = u_k(x)$, with $\mathcal H_{u_k}$ the inverse mapping given by the partial Legendre transform,
and obtain
\begin{equation}\label{SPbar}\begin{cases}
 \mathcal F( D^2 \mathcal H_{u_k}, \nabla \mathcal H_{u_k}) =0  \quad \text{in $B_{1/2} \cap  \{y_n>0\}$},\\
g( \nabla \mathcal H_{u_k})=0, \quad \text{on $B_{1/2} \cap \{y_n=0\}$}.
\end{cases}\end{equation}The free boundary of $u_k$ is given by the graph of the trace of $\mathcal H_{u_k}$ on $\{y_n=0\}$. 

 The difference $\psi:=\mathcal H_{u_2} -\mathcal H_{u_1} \geq 0$ will then satisfy an equation of the form
$$a^{ij}\p_{ij}\psi +b^i \p_i \psi=0 \quad \text{in $B_{1/2} \cap \{y_ n>0\}$}$$ with boundary condition,
$$\psi_n= d^i \p_i \psi \quad \text{on $B_{1/2} \cap \{y_n=0\}$},$$ 
and $a^{ij}, b^i$ H\"older continuous and $d^i \in C^{1,\alpha}$ with norms depending on the $C^{2,\alpha}$ norm of $\mathcal H_{u_k}$ which in turn depends on the $C^{2,\alpha}$ norm of $u_k$ (hence is bounded by a universal constant). By the standard regularity estimates in Theorem \ref{hn}, the $C^{2,\alpha}$ norm of $\psi$ is controlled by its $L^\infty$ norm. Moreover by the Harnack inequality, Theorem \ref{hn}, $\psi \sim \|\psi\|_\infty$ in $B_{1/4} \cap \{y_n>0\}$. Thus,  
$$u_1(x) = u_2(x', x_n- \psi(x', u_1(x))), \quad x\in B_{1/8} \cap \overline{\{u_1>0\}},$$
for $\psi$ as above, from which our claim follows.
\end{proof}

We conclude the Appendix by stating the Schauder estimates and the Harnack inequality that we need for boundary problems.

\begin{thm}\label{hn} Let $v$ satisfy,
$$a^{ij}\p_{ij}v +b^i \p_i v + c v=f \quad \text{in $B_1 \cap \Omega,$}$$
$$ \beta \cdot \nabla v + h v= g \quad \text{on $B_1 \cap \p \Omega$},$$
with $a^{ij}$ uniformly elliptic, and $\beta \cdot \nu \geq c_0>0,$ and 
$$\Omega:= \{x_n> \phi(x')\}, \quad \phi(0)=0.$$
\begin{enumerate}
\item If $a^{ij}, b^i, c, f \in C^{0,\alpha},$ $\beta, h, g \in C^{1,\alpha},$ $\phi \in C^{2,\alpha},$ then $$\|v\|_{C^{2,\alpha}(B_{1/2} \cap \Omega)} \leq C (\|v\|_\infty + \|f\|_{0,\alpha}+\|g\|_{1,\alpha});$$
\item if $v \geq 0$, $a^{ij}, b^i, c, \beta, h \in C^{0,\alpha},$ $\phi \in C^1,$ then 
$$\sup_{B_{1/2} \cap \Omega} v \leq C (\inf_{B_{1/2}  \cap \Omega}v + \|f\|_\infty+\|g\|_\infty),$$ for $C>0$ depending on $n,$ the coefficients, the ellipticity constants, and $\Omega.$\end{enumerate}
\end{thm}
Theorem \ref{hn} is contained for example in \cite{L}. However $(ii)$ is stated under a sign assumption on the lower order coefficients, that is $c, h \leq 0$, which is needed for the existence theory. We point out that these assumptions are not needed for the 
Harnack inequality, since after a dilation it suffices to prove the result when the lower order coefficients are arbitrarily small, and the standard arguments continue to hold. Since we could not find an explicit reference, we provide below a sketch of the proof of the one-sided Harnack inequality which we used in Proposition \ref{vas}.

\begin{lem} Let $v \geq 0$ satisfy 
$$\Delta v=0 \quad \text{in $B_1 \cap \Omega$},$$
$$\p_\nu v + h v =g, \quad \text{in $B_1 \cap \p \Omega$},$$ with $\Omega$ as in Theorem $\ref{hn},$ and $\nu$ the inner unit normal to $\p \Omega$. If $v(\bar x) \geq 1$ at $\bar x = \frac 1 4 e_n,$ then
$$\inf_{B_{1/2 \cap \Omega}} v \geq c- C \|g\|_\infty,$$ for $c, C>0$ depending on $n, \|h\|_\infty, \Omega$.

\end{lem}
\begin{proof} After a dilation, we may reduce to the case when $\|h\|_\infty, \|g\|_\infty, \|\phi\|_{C^1} \leq \eps_0$ for some $\eps_0$ small, and we need to show that
$$v \geq \frac c 2, \quad \text{in $B_{r_0} \cap \Omega$, for some small $r_0$}.$$ 

By the interior Harnack inequality, $v \geq c_0$ on $B_{1/8}(\bar x)$.
Let
$$V:= c_1(|x-\bar x|^{2-n} - (\frac 1 2)^{2-n})^+, \quad \text{in $B_{1/2}(\bar x)$}$$ with $c_1:=c_0((1/8)^{2-n} - (1/2)^{2-n})^{-1}.$
We claim that,
$$v \geq V \quad \text{on $B_{1/2}(\bar x) \cap \Omega$},$$ from which the desired bound follows.
Indeed, let $w:= v- V$ and assume $w(x_0):= \min_{\bar B_{1/2}(\bar x) \cap \bar\Omega}w <0$. By the maximum principle, $x_0$ cannot occur in $(B_{1/2}(\bar x) \setminus B_{1/8}(\bar x)) \cap \Omega$. Thus, we only need to rule out that $x_0 \in B_{1/2}(\bar x) \cap \p \Omega$. On the other hand, at such point $w_\nu(x_0) \geq 0$, hence $v_\nu(x_0) \geq V_\nu(x_0) \geq c_2>0,$ if $\eps_0$ is chosen small. Therefore,
$$\eps_0 + \eps_0V(x_0) \geq \eps_0 + \eps_0 v(x_0)  \geq g(x_0) -h(x_0) v(x_0)=\p_\nu v(x_0) \geq c_2,$$
and we reach a contradiction by choosing $\eps_0$ small enough.
\end{proof}

\end{document}